\documentclass[11pt]{amsart}
\usepackage[T1]{fontenc}
\usepackage{lmodern}
\usepackage[margin=1in]{geometry}
\usepackage[nopatch=item]{microtype}
\usepackage[hidelinks]{hyperref}
\usepackage{needspace}
\numberwithin{equation}{section}
\newtheorem{theorem}{Theorem}[section]
\newtheorem{proposition}[theorem]{Proposition}
\newtheorem{lemma}[theorem]{Lemma}
\newtheorem{corollary}[theorem]{Corollary}
\theoremstyle{remark}
\newtheorem{remark}[theorem]{Remark}

\DeclareMathOperator{\tr}{tr}
\DeclareMathOperator{\divg}{div}
\DeclareMathOperator{\Hess}{Hess}
\DeclareMathOperator{\id}{id}
\DeclareMathOperator{\ran}{ran}
\newcommand{\gb}{\bar g}
\newcommand{\Lie}{\mathcal L}
\newcommand{\dd}{\,d}
\newcommand{\Sph}{\mathbb S}
\newcommand{\R}{\mathbb R}
\newcommand{\norm}[1]{\lVert #1\rVert}
\newcommand{\abs}[1]{\lvert #1\rvert}
\newcommand{\ip}[2]{\langle #1,#2\rangle}
\newcommand{\DR}{D R_{\gb}}
\newcommand{\DHg}{D H_{\gb}}
\newcommand{\DRs}{D R_{\gb}^{*}}
\allowdisplaybreaks[2]
\title[Scalar-curvature rigidity of spherical caps]{Sharp Local Scalar-Curvature
Rigidity of Spherical Caps}
\author{Liuwei Gong}
\address{Department of Mathematics, The Chinese University of Hong Kong,
Shatin, NT, Hong Kong}
\email{lwgong@math.cuhk.edu.hk}
\author{Juncheng Wei}
\address{Department of Mathematics, The Chinese University of Hong Kong,
Shatin, NT, Hong Kong}
\email{wei@math.cuhk.edu.hk}
\date{}
\hypersetup{
pdftitle={Sharp Local Scalar-Curvature Rigidity of Spherical Caps},
pdfauthor={Liuwei Gong and Juncheng Wei},
pdfsubject={Local scalar-curvature rigidity of proper spherical caps}
}
\begin{document}
\begin{abstract}
Let $B(\delta)\subset\mathbb S^n$ be a geodesic ball with the round metric
$\bar{g}$ of sectional curvature one, where $n\geq3$ and
$0<\delta<\pi/2$. We prove local scalar-curvature rigidity under preservation
of the induced boundary metric and a one-sided lower bound for the boundary
mean curvature. More precisely, every sufficiently small $C^{2,\alpha}$
perturbation $g$ of $\bar{g}$ satisfying
$R(g)\geq n(n-1)$ and $H(g)\geq(n-1)\cot\delta$ is round up to a
diffeomorphism fixing the boundary pointwise.

Inspired by the scalar--vector--tensor decomposition used in the study
of Yamabe compactness, we establish positivity of the Brendle--Marques
quadratic form in the divergence-free gauge for variations preserving the induced boundary
metric and satisfying the joint linearized constraints
$D R_{\bar{g}}(h)=D H_{\bar{g}}(h)=0$. The hemisphere counterexamples of
Brendle--Marques--Neves show that the radius range is sharp.
\end{abstract}
\maketitle

\section{Introduction}
Scalar-curvature rigidity asks whether a lower bound for scalar curvature,
together with prescribed boundary geometry, determines a Riemannian
metric. For a round geodesic ball in the unit sphere, the natural boundary
data are the induced metric and the mean curvature. Min-Oo's conjecture for the hemisphere
\cite{minoo} was disproved by Brendle, Marques, and Neves \cite{bmn}.
For sufficiently small balls, local rigidity under a lower bound for the
boundary mean curvature was proved by Brendle and Marques \cite{bm}
and extended to a larger range of radii by Cox, Miao, and Tam \cite{cmt}.

Let $\gb$ denote the metric of sectional curvature one on $\Sph^n$, and set
\[
\Omega=B(\delta)=\{x\in\Sph^n:x_{n+1}\geq\cos\delta\},
\qquad \Sigma=\partial B(\delta).
\]
We call $B(\delta)$ a \emph{proper cap} when $0<\delta<\pi/2$.
Let $r$ be distance from its center and put $\lambda=x_{n+1}=\cos r$
and $c=\cos\delta$. Write $\DR$ and $\DHg$ for the linearizations
of scalar and mean curvature at $\gb$.
For a metric $g$ on $\Omega$, write $g^T$ for its induced metric on $\Sigma$
and $H(g)$ for the outward mean curvature. The problem considered here is
whether
\begin{equation}\label{eq:admissible}
R(g)\geq n(n-1)\quad\hbox{in }\Omega,
\qquad g^T=\gb^T,
\qquad H(g)\geq(n-1)\cot\delta\quad\hbox{on }\Sigma
\end{equation}
forces $g$ to be round when $g$ is close to $\gb$. We call a metric
satisfying \eqref{eq:admissible} \emph{admissible}. All curvature
inequalities are pointwise. A diffeomorphism is \emph{boundary-fixing}
if its restriction to $\Sigma$ is the identity.

Brendle and Marques \cite{bm} proved this local rigidity statement under the
radius restriction
\begin{equation}\label{eq:bm-range}
\cos\delta\geq\frac{2}{\sqrt{n+3}}.
\end{equation}
Their argument first uses a boundary-fixing diffeomorphism to make the
metric perturbation divergence-free. It then estimates a quadratic form
obtained by integrating the scalar-curvature defect against the static
potential $\cos r$. However, the boundary coefficients of this form can be negative when \eqref{eq:bm-range} fails. Cox, Miao, and Tam
\cite{cmt} improved the admissible radius range by controlling these terms
with additional integral and spectral estimates. The question of the
largest locally rigid spherical ball was also raised in Cox's dissertation
\cite{cox}. Our main result covers the full range $0<\delta<\pi/2$.

\begin{theorem}\label{thm:main}
Let $n\geq3$, $0<\delta<\pi/2$, and $0<\alpha<1$. There is
$\varepsilon=\varepsilon(n,\delta,\alpha)>0$ with the following property.
If $g$ is a $C^{2,\alpha}$ metric on $B(\delta)$ satisfying
\eqref{eq:admissible} and
\[
\norm{g-\gb}_{C^{2,\alpha}(B(\delta),\gb)}<\varepsilon,
\]
then there is a $C^{3,\alpha}$ diffeomorphism $\phi:B(\delta)\to B(\delta)$
such that
\[
\phi|_\Sigma=\id,\qquad \phi^*g=\gb.
\]
\end{theorem}

The radius range is sharp: the smooth perturbations in
\cite[Corollary~15]{bmn} show that local rigidity fails at $\delta=\pi/2$.
We distinguish local rigidity from \emph{global rigidity}, which requires every
smooth admissible metric on the cap to be round, without a smallness
assumption. Let $r_{\mathrm{loc}}(n)$ and $r_{\mathrm{glob}}(n)$ be the
suprema of the radii $a\leq\pi/2$ for which the respective rigidity
statement holds for every $0<\delta<a$. Theorem~\ref{thm:main},
\cite[Theorem~7]{bmn}, and \cite[Corollary~1.3]{zhangzhao} give
\[
0<r_{\mathrm{glob}}(n)<r_{\mathrm{loc}}(n)=\frac\pi2.
\]
The positive global lower bound follows from the theorem of Zhang and
Zhao \cite{zhangzhao} for sufficiently small spherical caps. Their result
allows all smooth compact connected oriented fill-ins with the prescribed
boundary metric and curvature lower bounds, without a smallness assumption.
Theorem~\ref{thm:main} treats small perturbations on every proper cap.

The strict global upper bound follows by restricting the metric in
\cite[Theorem~7]{bmn} to proper caps whose boundaries lie in its round
equatorial collar. Sweeney's tunnel construction \cite{sweeneyminoo}
can be carried out in a region of strict scalar-curvature increase,
preserving this collar and producing further global counterexamples.
These collar and tunnel constructions do not yield arbitrarily small
$C^{2,\alpha}$ perturbations on a fixed proper cap; they therefore do not
imply $r_{\mathrm{loc}}(n)<\pi/2$. The exact global radius remains open.

Theorem~\ref{thm:main} allows general metric perturbations subject to
\eqref{eq:admissible} and smallness. Several earlier spherical rigidity
theorems use additional geometric structure. Hang and Wang \cite{hangwang} proved
hemisphere rigidity within the round conformal class, and Kim \cite{kim}
studied conformal scalar-curvature rigidity using a Dirichlet spectral
condition. Spiegel \cite{spiegel} and Barbosa, Cavalcante, and Espinar
\cite{bce} obtained rigidity of spherical caps under local conformal flatness and
umbilic-boundary hypotheses. Miao and Tam \cite{miaotam} proved a local
hemisphere theorem with an additional volume constraint. In dimension
three, Hu, Liu, and Shi \cite{hls} established a global cap theorem under the
extra tensor inequality $g\geq\gb$. Lee and Tam \cite{leetam} proved
rigidity for smooth spherical domains in every dimension under $g\geq\gb$ and equality of the metrics along the boundary,
together with the scalar- and mean-curvature comparisons.

The static structure is also central in more general domains. Qing and
Yuan \cite{qingyuan} proved local rigidity for sufficiently small balls in
vacuum static spaces around points where the static potential is nonzero.
Ho and Huang \cite{hohuang}, and Sheng \cite{shengdeform,shengstatic},
studied the joint prescription of scalar and boundary mean curvature and
the obstructions given by the adjoint of its linearization. Yuan \cite{yuan} related conformal rigidity and
compactly supported deformations to a Dirichlet spectral condition and
Brown--York mass. Ge's conjecture for proper isoparametric tubes
\cite{ge} includes the geodesic-ball problem as its degree-one case. These
works provide a broader geometric setting for the present theorem.

\begin{remark}
Brendle--Marques--Neves noted the relation of their
perturbation argument to Yamabe noncompactness after Theorem~4 of
\cite{bmn}.
The scalar--vector--tensor decomposition used below is motivated by the
study of Pohozaev forms in Yamabe compactness \cite[Appendix~A]{kms},
the boundary Yamabe problem \cite{gkmwblowup} and
\cite[Section~6 and Appendix~A]{corderomisteli},
and the constant $Q$-curvature problems
\cite[Section~6.2 and Appendix~D]{gkwq}.
For $h^T=0$ and $\DR(h)=\DHg(h)=0$, the Brendle--Marques form
$Q$ is the weighted curvature second variation in
\eqref{eq:critical-hessian}. Its relation to the Pohozaev form is
\begin{equation}\label{eq:intro-pohozaev}
\begin{aligned}
\mathcal P(h,h)&=-\frac{(n-2)^2}{2^{n+1}(n-1)}Q(\widehat h),\\
\widehat h&=4U^{4/(n-2)}
\left(h+\frac{4z[h]}{(n-2)U}g_{\mathrm{euc}}\right).
\end{aligned}
\end{equation}
Here $U=(1+|y|^2)^{-(n-2)/2}$, $h$ is the trace-free Euclidean metric variation,
and $z[h]$ is its linear conformal correction. On the right, $Q$ is
evaluated on the hemisphere $\Omega_+=\{x\in\Sph^n:x_n\geq0\}$, with
the same static function $\lambda=x_{n+1}$. The definition of the
Pohozaev form, the smoothness assumptions, and the boundary terms for
a finite radius are given in Appendix~\ref{app:pohozaev}.

Thus both forms come from the same weighted second variation. The Yamabe correction
imposes $\DR(\widehat h)=\DHg(\widehat h)=0$; rigidity also fixes the
induced boundary metric and uses divergence gauge. The domains differ:
$\lambda>0$ on the proper cap, whereas $\lambda$ changes sign on
$\Omega_+$.
\end{remark}

The proof uses the static identity to identify the linearized constraints
on which positivity is needed. If a smooth symmetric tensor $h$ satisfies
$h^T=0$, integration by parts against $\lambda$ gives
\begin{equation}\label{eq:intro-static}
\int_\Omega\lambda\DR(h)\dd v
+2c\int_\Sigma\DHg(h)\dd\sigma=0.
\end{equation}
If a differentiable path of admissible metrics starts at $\gb$, its
initial velocity satisfies $h^T=0$, $\DR(h)\geq0$, and $\DHg(h)\geq0$.
On a proper cap, both weights in \eqref{eq:intro-static} are strictly
positive. The identity therefore forces
\[
\DR(h)=0,\qquad \DHg(h)=0.
\]
We call such a tensor, with $h^T=0$, \emph{critical}. On these tensors,
$Q$ is $-1/2$ times the Hessian of $I$, as in \eqref{eq:critical-hessian},
and is unchanged by
adding Lie derivatives generated by vector fields vanishing on the boundary. The restriction to this joint kernel is the
additional information used to control the boundary terms beyond the
range \eqref{eq:bm-range}.

The key geometric step is a change of gauge adapted to the spherical
harmonic decomposition. In each scalar or vector mode, we subtract
$\Lie_X\gb$, with $X|_\Sigma=0$, to obtain
\[
h-\Lie_X\gb=2u\,dr^2+dr\otimes v+v\otimes dr.
\]
Here $u$ is a function and $v$ is a one-form tangent to the concentric
spheres. The tangential--tangential components vanish on each sphere. This gauge
differs from the divergence-free gauge used in the nonlinear expansion:
the transformed tensor need not be divergence-free. The gauge invariance
of the constrained Hessian justifies this change, which reduces the joint
curvature constraints to a first-order radial relation. For scalar degrees
at least two, one radial function determines the tensor, and
Proposition~\ref{prop:scalar} expresses $Q$ as two positive
integrals. The boundary condition and smoothness at the center eliminate
the endpoint terms. The low scalar degrees and the vector equality cases
are gauge tensors; the tensor modes are strictly positive.

The static weight has boundary value $\cos\delta$, which tends to zero as
$\delta\uparrow\pi/2$. We use compactness on each fixed proper cap,
without requiring a coercivity estimate for the Brendle--Marques form
uniform in the radius. If local rigidity fails, there are admissible
metrics $\gb+h_j$ in divergence gauge, with $h_j\ne0$ and
$h_j\to0$ in $C^{2,\alpha}$. After normalization and passage to a subsequence,
\[
v_j=\frac{h_j}{\norm{h_j}_{H^1}},\qquad
\norm{v_j}_{H^1}=1,\qquad v_j\rightharpoonup v\quad\text{in }H^1.
\]
This normalization supplies no uniform bounds for higher derivatives,
so smoothness of $v$ cannot initially be assumed.

We therefore need positivity for $H^1$ tensors satisfying the constraints
weakly. In each spherical mode, the trace solves a scalar
elliptic equation. Subtracting its Hessian leaves a trace-free,
divergence-free tensor whose radial coefficients can be approximated
smoothly while preserving all constraints and boundary values.
The positive factorizations and their equality cases then pass to the
weak space.
Taylor expansion and the positive static potential show that $v$
satisfies the weak constraints.
The Brendle--Marques expansion and weak positivity force $v$ to
vanish, whereas the G\aa rding inequality rules out convergence to zero.

Section~\ref{sec:prelim} records the conventions and standard analytic
identities. Section~\ref{sec:geometry} computes the Hessian on critical
tensors, including the changes of gauge and their smoothness at the center.
Section~\ref{sec:weak} establishes positivity on the weak constraint space.
Section~\ref{sec:nonlinear} proves Theorem~\ref{thm:main}.
Appendix~\ref{app:pohozaev} proves the comparison with the Pohozaev form,
including the boundary terms at a finite radius.

Unless otherwise specified, traces, divergences, norms, covariant
derivatives, and volume forms are taken with respect to $\gb$.
Our convention is $\Delta=\tr\Hess$, so the spherical Laplacian has
nonpositive eigenvalues. Unless stated otherwise, $n\geq3$ and
$0<\delta<\pi/2$ are fixed throughout the proof.
Constants may depend on $n$, $\delta$, and the regularity exponent.

\medskip
\noindent\textbf{Acknowledgements.}
The authors gratefully acknowledge support from the Research Grants
Council of Hong Kong (RGC) through the project
\emph{On Fujita equation in the critical or supercritical regime}.

ChatGPT and Codex (OpenAI) were used to revise the exposition and LaTeX presentation,
simplify the variational and nonlinear arguments, and check selected
calculations and bibliographic details.
The authors are responsible for the mathematical content and references.

\section{Preliminaries}\label{sec:prelim}

\subsection{Geometry and conventions}
In geodesic polar coordinates,
\[
\gb=dr^2+\sin^2r\,\gamma,\qquad 0\leq r\leq\delta,
\]
where $\gamma$ is the round metric on $\Sph^{n-1}$. We write
$dv=dv_{\gb}$ and $d\sigma=d\sigma_{\gb^T}$. Indices $i,j$ refer to
directions tangent to $\Omega$, and $a,b$ to tangential directions on a concentric sphere.

For the inclusion $\iota:\Sigma\hookrightarrow\Omega$, set
\[
g^T=\iota^*g,\qquad
\mathrm{II}_g(U,V)=g(\nabla^g_U\nu_g,V),\qquad H(g)=\tr_{g^T}\mathrm{II}_g,
\]
where $\nu_g$ is the outward unit normal. Thus
\[
\nu=\nu_{\gb}=\partial_r,\qquad \mathrm{II}_{\gb}=\cot\delta\,\gb^T,
\qquad H(\gb)=(n-1)\cot\delta.
\]
For a symmetric tensor $h$, write
\[
\tau=\tr h,\qquad h_{\nu\nu}=h(\nu,\nu),\qquad
X_h=\bigl(h(\nu,\cdot)^T\bigr)^\sharp.
\]
We write $\tr_\Sigma h=\tr_{\gb^T}h^T$. The condition $h^T=0$ fixes the tangential--tangential components, while
$h_{\nu\nu}$ and $X_h$ may be nonzero. In particular,
$\tau=h_{\nu\nu}$ on $\Sigma$ when $h^T=0$.

An infinitesimal boundary-fixing change of coordinates has the form
\[
\Lie_X\gb,\qquad X|_\Sigma=0,
\qquad (\Lie_X\gb)_{ij}=\nabla_iX_j+\nabla_jX_i.
\]
We refer to these tensors as \emph{gauge tensors}.

A smooth tensor is \emph{critical} if
\[
h^T=0,\qquad \DR(h)=0,\qquad \DHg(h)=0.
\]

\subsection{Linearizations and the static identity}
The scalar- and mean-curvature linearizations at $\gb$ are
\begin{align}
\DR(h)&=-\Delta\tau+\divg\divg h-(n-1)\tau,
\label{eq:DR}\\
2\DHg(h)&=(d\tau-\divg h)(\nu)-\divg_\Sigma X_h
-\ip{\mathrm{II}_{\gb}}{h^T}.
\label{eq:DH}
\end{align}
The formal scalar-curvature adjoint is
\[
\DRs \varphi=\Hess \varphi-(\Delta \varphi)\gb-(n-1)\varphi\gb.
\]
The static potential and its boundary value are
\[
\lambda=\cos r,\qquad c=\cos\delta.
\]
They satisfy
\[
\Hess\lambda=-\lambda\gb,\qquad \Delta\lambda=-n\lambda,
\qquad \DRs\lambda=0,\qquad \partial_\nu\lambda=-\sin\delta.
\]

\begin{lemma}\label{lem:green}
For every smooth $h$ and every smooth $\varphi$ on the closed cap,
\begin{align}\label{eq:green-full}
\int_\Omega \varphi\DR(h)\dd v+2\int_\Sigma \varphi\DHg(h)\dd\sigma
={}&\int_\Omega\ip{\DRs \varphi}{h}\dd v\\
&+\int_\Sigma\bigl((\partial_\nu \varphi)\tr_\Sigma h
-\varphi\ip{\mathrm{II}_{\gb}}{h^T}\bigr)\dd\sigma.\notag
\end{align}
In particular, if $h^T=0$, then
\begin{equation}\label{eq:green-fixed}
\int_\Omega \varphi\DR(h)\dd v+2\int_\Sigma \varphi\DHg(h)\dd\sigma
=\int_\Omega\ip{\DRs \varphi}{h}\dd v.
\end{equation}
Taking $\varphi=\lambda$ gives \eqref{eq:intro-static}.
\end{lemma}
\begin{proof}
Twice integrating \eqref{eq:DR} by parts gives
\[
\int_\Omega \varphi\DR(h)\dd v-\int_\Omega\ip{\DRs \varphi}{h}\dd v
=\int_\Sigma\left[\varphi(\divg h-d\tau)(\nu)
+\tau\partial_\nu \varphi-h(\nabla \varphi,\nu)\right]\dd\sigma.
\]
Add $2\int_\Sigma \varphi\DHg(h)\,d\sigma$ and use
\[
-\int_\Sigma \varphi\divg_\Sigma X_h\dd\sigma
=\int_\Sigma h(\nabla_\Sigma \varphi,\nu)\dd\sigma,
\qquad \tau-h_{\nu\nu}=\tr_\Sigma h.
\]
This gives \eqref{eq:green-full}.
\end{proof}

\begin{corollary}
For a smooth tensor with $h^T=0$ on a proper cap,
\[
\DR(h)\geq0,\quad \DHg(h)\geq0
\quad\Longrightarrow\quad \DR(h)=\DHg(h)=0.
\]
This implication fails on the hemisphere.
\end{corollary}
\begin{proof}
For $\delta<\pi/2$, the two integrands in \eqref{eq:intro-static} are
nonnegative and $\lambda\geq c>0$. They therefore vanish. At $\delta=\pi/2$,
the tensor $h=-2\lambda\gb$ has $h^T=0$, while
\[
\DR(h)=2(n-1)(\Delta\lambda+n\lambda)=0,
\qquad \DHg(h)=-(n-1)\partial_\nu\lambda=n-1>0.
\]
\end{proof}

\subsection{Divergence gauge and the quadratic expansion}
We use the boundary-fixing change of coordinates in
\cite[Proposition~11]{bm}, in H\"older spaces. The stated regularity and
estimate follow from the boundary Schauder estimates for the equation
imposing $\divg h=0$, with the vector field zero on $\Sigma$ \cite{adn}.
\begin{lemma}\label{lem:slice}
If $g^T=\gb^T$ and $g$ is sufficiently close to $\gb$ in $C^{2,\alpha}$,
there is a boundary-fixing $C^{3,\alpha}$ diffeomorphism $\phi$ such that
\[
h=\phi^*g-\gb,\qquad \divg h=0,\qquad h^T=0,
\qquad \norm{h}_{C^{2,\alpha}}\leq C\norm{g-\gb}_{C^{2,\alpha}}.
\]
\end{lemma}

For $g=\gb+h$ inducing $\gb^T$, put
\begin{equation}\label{eq:defect}
D(g)=\int_\Omega\lambda\bigl(R(g)-n(n-1)\bigr)\dd v
+\int_\Sigma(2-h_{\nu\nu})\lambda
\bigl(H(g)-H(\gb)\bigr)\dd\sigma.
\end{equation}
All measures in $D$ are fixed round measures. For smooth $h$ with
$h^T=0$, define
\begin{equation}\label{eq:Q-definition}
Q(h)=-\frac12\left.\frac{d^2}{dt^2}\right|_{t=0}D(\gb+th).
\end{equation}
When $\divg h=0$, the Brendle--Marques calculation \cite{bm,cmt} gives
\begin{align}\label{eq:Q}
Q(h)={}&\frac14\int_\Omega\lambda
(\abs{\nabla h}^2+\abs{\nabla\tau}^2)\dd v
+\frac12\int_\Omega\lambda(\abs h^2+\tau^2)\dd v\\
&+\frac{(n+3)c^2-4}{4\sin\delta}\int_\Sigma h_{\nu\nu}^2\dd\sigma
+\frac{(n+1)c^2-1}{2\sin\delta}\int_\Sigma\abs{X_h}^2\dd\sigma.
\notag
\end{align}
This formula extends $Q$ continuously to divergence-free $H^1$ tensors
with $h^T=0$.

\begin{lemma}
If $h\in C^2$, $\divg h=0$, $h^T=0$, and $\norm h_{C^1}$ is sufficiently small, then
\begin{equation}\label{eq:BM}
\abs{D(\gb+h)+Q(h)}\leq C\norm h_{C^1}\norm h_{H^1}^2.
\end{equation}
If $\gb+h$ is admissible and $\abs{h_{\nu\nu}}<1$ on $\Sigma$, then
$D(\gb+h)\geq0$.
\end{lemma}
\begin{proof}
The expansion is \cite[Theorem~10]{bm}, equivalently
\cite[Theorem~2.1]{cmt}. In their boundary terms,
\[
\partial_\nu\lambda+\frac14H(\gb)\lambda
=\frac{(n+3)c^2-4}{4\sin\delta},\qquad
\frac12\partial_\nu\lambda+\frac{n}{2(n-1)}H(\gb)\lambda
=\frac{(n+1)c^2-1}{2\sin\delta}.
\]
The cubic interior and boundary remainders are bounded by the right side of
\eqref{eq:BM}, using the trace theorem. The final assertion follows directly
from \eqref{eq:defect}.
\end{proof}

\begin{lemma}
For each proper cap there are $\kappa>0$ and $C<\infty$ such that
\begin{equation}\label{eq:garding}
Q(h)\geq\kappa\norm h_{H^1}^2-C\norm h_{L^2}^2
\qquad(h\in H^1,\ \divg h=0,\ h^T=0).
\end{equation}
\end{lemma}
\begin{proof}
Since $\lambda\geq c>0$, \eqref{eq:Q} gives
\[
Q(h)\geq\frac c4\norm{\nabla h}_{L^2(\Omega)}^2
-C\norm h_{L^2(\Sigma)}^2.
\]
For every $\eta>0$, the trace inequality gives
\[
\norm h_{L^2(\Sigma)}^2
\leq\eta\norm{\nabla h}_{L^2(\Omega)}^2
+C_\eta\norm h_{L^2(\Omega)}^2.
\]
Choosing $C\eta\leq c/8$ yields
$Q(h)\geq(c/8)\norm{\nabla h}_{L^2}^2-C\norm h_{L^2}^2$,
which proves \eqref{eq:garding} after changing $C$.
\end{proof}

\section{The constrained Hessian}\label{sec:geometry}
The divergence-free formula \eqref{eq:Q} is suited to the nonlinear
curvature expansion. For the spherical calculation, we use a tensor whose
tangential components vanish on every concentric sphere. We justify this
change of gauge by expressing $Q$ as a Hessian on critical tensors.

\subsection{Spherical modes}\label{subsec:modes}
For the geometric calculations put $m=n-1$ and $s=\sin r$; thus
$\gamma_r=s^2\gamma$ is the metric on a concentric sphere. Primes denote
differentiation with respect to $r$; derivatives along a metric path are
written as $\partial_t$, with $t$ reserved for the path parameter. The symbol $D$
in angular components denotes the connection of $\gamma$. A scalar harmonic
of degree $\ell$ satisfies
\[
-\Delta_\gamma Y=\mu Y,\qquad \mu=\ell(\ell+m-1).
\]
For $\ell\geq2$, write
$Y^\circ_{ab}=D_aD_bY+(\mu/m)Y\gamma_{ab}$. This tensor vanishes in degrees
$0$ and $1$. A scalar mode has components
\begin{equation}\label{eq:scalar-mode}
h_{rr}=a(r)Y,\qquad h_{ra}=b(r)D_aY,\qquad
h_{ab}=s^2\bigl(p(r)Y\gamma_{ab}+f(r)Y^\circ_{ab}\bigr).
\end{equation}
A coexact one-form harmonic $V$ of degree $\ell\geq1$ satisfies
\[
D^aV_a=0,\qquad -D^aD_aV=(\mu-1)V.
\]
Write $V^\sharp$ for its $\gamma$-dual. A vector mode has the form
\begin{equation}\label{eq:vector-mode}
h_{rr}=0,\qquad h_{ra}=b(r)V_a,\qquad
h_{ab}=s^2f(r)(D_aV_b+D_bV_a).
\end{equation}
A tensor mode is
\begin{equation}\label{eq:tensor-mode}
h_{rr}=h_{ra}=0,\qquad h_{ab}=s^2f(r)T_{ab},
\qquad D^aT_{ab}=0,\quad \gamma^{ab}T_{ab}=0,
\end{equation}
where $T$ is a tangential trace-free divergence-free eigentensor. Such
summands are included only when they exist.

Choose orthonormal bases in the scalar, coexact-vector, and tensor harmonic
spaces. A \emph{block} consists of one degree and one of these three types,
including all angular multiplicities and arbitrary radial coefficients.
Write $\Pi_\beta$ for its projection. Distinct normalized angular basis elements in a block can be treated
separately. The full orthogonality and the behavior of these projections on
$H^1$ tensors will be proved in Lemma~\ref{lem:blocks}.

The boundary condition $h^T=0$ is $p(\delta)=f(\delta)=0$ in scalar type
and $f(\delta)=0$ in non-Killing vector and tensor modes.
The condition on $f$ is absent in scalar degrees $0,1$, and the vector condition
is absent when $V$ is Killing, since its tangential tensor is then zero.
For one-forms $\alpha,\beta$, we use
$\alpha\odot\beta=(\alpha\otimes\beta+\beta\otimes\alpha)/2$.
A superscript $\circ$ on a tangential tensor denotes its $\gamma_r$-trace-free
part.

\subsection{Identification and gauge invariance}
On metrics inducing $\gb^T$, consider the functional
\[
I(g)=\int_\Omega\lambda\bigl(R(g)-n(n-1)\bigr)\dd v_g
+2c\int_\Sigma\bigl(H(g)-H(\gb)\bigr)\dd\sigma.
\]
Unlike $D$, its interior measure is the metric volume form. By
Lemma~\ref{lem:green}, $DI(\gb)[h]=0$ for every $h$ with $h^T=0$.

\Needspace{10\baselineskip}
\begin{lemma}\label{lem:comparison}
For smooth $h$ with $h^T=0$,
\begin{equation}\label{eq:comparison}
Q(h)=-\frac12D^2I(\gb)[h,h]
+\frac12\int_\Omega\lambda\tau\DR(h)\dd v
+c\int_\Sigma h_{\nu\nu}\DHg(h)\dd\sigma.
\end{equation}
For critical $h$, therefore,
\begin{equation}\label{eq:critical-hessian}
Q(h)=-\frac12D^2I(\gb)[h,h].
\end{equation}
If also $X|_\Sigma=0$, then
\begin{equation}\label{eq:Q-gauge}
Q(h+\Lie_X\gb)=Q(h).
\end{equation}
\end{lemma}
\begin{proof}
Along $g_t=\gb+th$, write
\[
R_j=\left.\partial_t^jR(g_t)\right|_{t=0},\qquad
H_j=\left.\partial_t^jH(g_t)\right|_{t=0}\quad(j=1,2).
\]
Thus $R_1=\DR(h)$ and $H_1=\DHg(h)$. The identities
\[
\left.\frac{d}{dt}\right|_0dv_{g_t}=\frac12\tau\,dv,
\qquad R(\gb)-n(n-1)=0
\]
give
\begin{align*}
\left.\frac{d^2}{dt^2}\right|_0
\int_\Omega\lambda\bigl(R(g_t)-n(n-1)\bigr)\dd v_{g_t}
&=\int_\Omega\lambda(R_2+\tau R_1)\dd v,\\
\left.\frac{d^2}{dt^2}\right|_0
\int_\Omega\lambda\bigl(R(g_t)-n(n-1)\bigr)\dd v
&=\int_\Omega\lambda R_2\dd v.
\end{align*}
The second derivatives of the respective boundary terms of $I$ and $D$
are $2c\int_\Sigma H_2\,d\sigma$ and
$2c\int_\Sigma(H_2-h_{\nu\nu}H_1)\,d\sigma$.
By \eqref{eq:Q-definition}, the second derivative of $D(\gb+th)$ at
$t=0$ is $-2Q(h)$. Subtracting proves
\eqref{eq:comparison}.

The fixed weight $\lambda$ prevents $I$ from being fully diffeomorphism
invariant. To check the restricted invariance, let $\psi_t$ be the flow of
$X$, with $X|_\Sigma=0$. Changing variables in the interior integral gives
\[
I(\psi_t^*g)=
\int_\Omega(\lambda\circ\psi_t^{-1})
\bigl(R(g)-n(n-1)\bigr)\dd v_g
+2c\int_\Sigma\bigl(H(g)-H(\gb)\bigr)\dd\sigma.
\]
Differentiating first in $t$ and then in the direction $h$ gives
\[
D^2I(\gb)[\Lie_X\gb,h]+DI(\gb)[\Lie_Xh]
=-\int_\Omega X(\lambda)\DR(h)\dd v=0.
\]
Since $X|_\Sigma=0$, we have $(\Lie_Xh)^T=0$, so the first-variation
term vanishes. Similarly, $I(\psi_t^*\gb)=0$ implies
\[
D^2I(\gb)[\Lie_X\gb,\Lie_X\gb]
=-DI(\gb)[\Lie_X^2\gb]=0.
\]
Since $X|_\Sigma=0$, the tensor $\Lie_X\gb$ is critical.
Thus $h+\Lie_X\gb$ is critical as well, and expanding
\eqref{eq:critical-hessian} proves \eqref{eq:Q-gauge}.
\end{proof}

\subsection{Removal of the tangential components}
We remove the tangential components of scalar and vector modes by a
boundary-fixing gauge change, first checking smoothness at the center
in normal coordinates.

For a scalar mode \eqref{eq:scalar-mode}, set
\begin{equation}\label{eq:gauge-coeff}
\eta=\frac f2,\qquad \xi=\frac{p+(\mu/m)f}{2\cot r}
\quad(\ell\geq2),
\end{equation}
and set $\eta=0$, $\xi=p/(2\cot r)$ for $\ell=0,1$.
\begin{lemma}\label{lem:pole}
For a smooth scalar mode, the vector field
\[
X=\xi Y\partial_r+\eta\nabla_\gamma Y
\]
is smooth at the center. For a smooth non-Killing vector mode,
$fV^\sharp$ is smooth there. If $V$ is Killing, then
\[
2\,dr\odot b(r)V\text{ smooth}\quad\Longrightarrow\quad
b(r)=r^3F(r^2),\qquad F\text{ smooth near }0.
\]
\end{lemma}
\begin{proof}
Write $x=r\omega$ in normal coordinates and let
$P_\ell(x)=r^\ell Y(\omega)$ be the homogeneous harmonic polynomial
associated with $Y$. Then
\begin{equation}\label{eq:cartesian-vector}
X=r^{2-\ell}\eta\nabla_{\R^n}P_\ell
+r^{-\ell-1}(\xi-\ell r\eta)P_\ell x.
\end{equation}
For $\ell\geq1$, smoothness at $x=0$ is equivalent to the conditions
\begin{equation}\label{eq:pole-vector}
\eta=r^{\ell-2}F(r^2),\qquad
\xi-\ell r\eta=r^{\ell+1}G(r^2).
\end{equation}
For $\ell=0$ the condition is $\xi=rF(r^2)$. These conditions follow by expanding the two polynomial fields in
\eqref{eq:cartesian-vector}. For a smooth field, Taylor's formula with
remainder gives the same conclusion.

The scalar-type polynomial tensors in normal coordinates are
\[
\Hess_{\R^n}P_\ell,\quad x^\flat\odot dP_\ell,
\quad P_\ell g_{\R^n},\quad P_\ell x^\flat\otimes x^\flat,
\]
with coefficients smooth in $r^2$. Only the first and third have nonzero
tangential--tangential restrictions. In a Euclidean unit tangential frame,
\[
(\Hess_{\R^n}P_\ell)^T
=r^{\ell-2}\left(Y^\circ-\frac{\ell(\ell-1)}mY\gamma\right).
\]
As $s/r$ is smooth and even, comparison with the given tensor yields
\[
f=r^{\ell-2}F_1(r^2),\qquad
p+\left(\frac\mu m-\ell\right)f=r^\ell G_1(r^2).
\]
Consequently
\[
\xi-\ell r\eta
=\frac{\tan r}{2}\left[p+\left(\frac\mu m-\ell\right)f\right]
+\frac\ell2(\tan r-r)f.
\]
The first term is $r^{\ell+1}$ times a smooth even function, and so is the
second because $\tan r-r=O(r^3)$. Thus \eqref{eq:pole-vector} holds.
In degrees $0$ and $1$, the same polynomial tensors give, respectively,
$p=F(r^2)$ and $p=rG(r^2)$. Multiplication by $\tan r/2$ proves the stated
low-degree regularity.

Here $\operatorname{sym}\nabla W$ denotes
the symmetric part of the Euclidean derivative of $W$.
For a coexact harmonic, let $W_\ell$ be its homogeneous tangential
polynomial vector field, with $W_\ell(r\omega)=r^\ell V(\omega)$ in
Euclidean components. The polar field $fV^\sharp$ becomes
$r^{1-\ell}fW_\ell$ in normal coordinates. The tangential polynomial
tensor $\operatorname{sym}\nabla W_\ell$ gives
$f=r^{\ell-1}F(r^2)$ whenever that tensor is nonzero. This proves the
non-Killing assertion. In the Killing case $W_1=Ax$ with $A$ skew, and
$\operatorname{sym}\nabla W_1=0$. The first nonzero mixed polynomial tensor
is $x^\flat\odot W_1^\flat$. Its polar expression is $r^3dr\odot V$.
Multiplying by a smooth function of $r^2$ gives $b=r^3F(r^2)$.
\end{proof}

\begin{lemma}\label{lem:relative}
Every smooth scalar or vector mode with $h^T=0$ differs from a smooth
boundary-fixing gauge tensor by a tensor of the form
\begin{equation}\label{eq:lapse-shift}
h=2u\,dr^2+2\,dr\odot v.
\end{equation}
Here $v$ is exact in scalar type and coexact in vector type.
\end{lemma}
\begin{proof}
For a scalar mode and $X=\xi Y\partial_r+\eta\nabla_\gamma Y$, direct
calculation gives
\[
(\Lie_X\gb)_{ab}
=s^2\left[\left(2\cot r\,\xi-\frac{2\mu}{m}\eta\right)
Y\gamma_{ab}+2\eta Y^\circ_{ab}\right].
\]
The choice \eqref{eq:gauge-coeff} removes the entire tangential part of
$h-\Lie_X\gb$. Since $p(\delta)=f(\delta)=0$ and $\cot\delta>0$, both
$\xi$ and $\eta$ vanish on $\Sigma$. Lemma~\ref{lem:pole} proves
smoothness at the center. In degrees $0,1$ the trace-free term is absent,
and $\xi=p/(2\cot r)$ suffices.

For a non-Killing vector mode, the tangential part of
$\Lie_{fV^\sharp}\gb$ is exactly the last expression in
\eqref{eq:vector-mode}. Subtracting it leaves only the mixed component.
The field is smooth by Lemma~\ref{lem:pole} and vanishes on the boundary
because $f(\delta)=0$. A Killing mode already has zero tangential part.
Tangential trace-free divergence-free tensors are orthogonal to both the
tensors $D^2Y$ and $D_aV_b+D_bV_a$, and are left
unchanged.
\end{proof}

After subtracting the gauge tensor, we again denote the representative by
$h$. Tensor modes retain the form \eqref{eq:tensor-mode}, with
$f(\delta)=0$. For \eqref{eq:lapse-shift}, we identify $v$ with
its $\gamma_r$-dual when writing a Lie derivative, and set
\[
d=\divg_{\gamma_r}v,\qquad w=m\cot r\,u+d.
\]
These are functions on the cap. Within a scalar mode, we also use $u$,
$d$, and $w$ for their radial coefficients.

\begin{lemma}
For a smooth tensor of the form \eqref{eq:lapse-shift},
\begin{align}
\DHg(h)&=-w(\delta),\label{eq:DH-radial}\\
\frac12\DR(h)
&=w'+(m+1)\cot r\,w-\Delta_{\gamma_r}u-ms^{-2}u.
\label{eq:DR-radial}
\end{align}
Consequently, the tensor is critical exactly when the right side of
\eqref{eq:DR-radial} vanishes and $w(\delta)=0$.
\end{lemma}
\begin{proof}
The relevant connection coefficients in polar coordinates are
$\Gamma^r_{ab}=-\cot r\,(\gamma_r)_{ab}$ and
$\Gamma^a_{rb}=\cot r\,\delta^a_b$. They give
\begin{equation}\label{eq:radial-div}
\tau=2u,\qquad
(\divg h)_r=2u'+d+2m\cot r\,u,\qquad
(\divg h)_a=v'_a+m\cot r\,v_a.
\end{equation}
At the outer boundary, $\divg_\Sigma X_h=d$.
Substitution in \eqref{eq:DH} gives \eqref{eq:DH-radial}.
To compute the scalar constraint, use
\[
\Delta u=u''+m\cot r\,u'+\Delta_{\gamma_r}u,
\qquad
\divg_{\gamma_r}(v')=d'+2\cot r\,d.
\]
Taking one more divergence in \eqref{eq:radial-div} and then using
\eqref{eq:DR} gives
\begin{align*}
\frac12\DR(h)
={}&d'+m\cot r\,u'+(m+1)\cot r\,d\\
&+\left(m(\cot r)'+m^2\cot^2r-m\right)u
-\Delta_{\gamma_r}u.
\end{align*}
Since $(\cot r)'=-s^{-2}$, the right side equals that of
\eqref{eq:DR-radial}.
\end{proof}

\subsection{A second-variation formula}
\begin{lemma}\label{lem:radial-hessian}
For a smooth critical tensor of the form \eqref{eq:lapse-shift},
\begin{equation}\label{eq:radial-hessian}
Q(h)=\int_\Omega\left\{
\frac\lambda4\abs{(\Lie_v\gamma_r)^\circ}^2
-\frac{m-1}{m}\lambda w^2-2\lambda'uw\right\}\dd v.
\end{equation}
\end{lemma}
\begin{proof}
Since $DI(\gb)=0$ on fixed-boundary variations, the Hessian depends only on
the initial velocity of a path. On an annulus $\varepsilon<r<\delta$, use
\begin{equation}\label{eq:ADM-path}
g_t=N_t^2dr^2+(\gamma_r)_{ab}
(dx^a+\beta_t^a dr)(dx^b+\beta_t^b dr),
\qquad N_t=1+tu,\quad \beta_t=t v^\sharp.
\end{equation}
Then $\partial_tg_t|_{t=0}=h$, and the induced metric on every concentric sphere is
fixed. Its unit normal and second fundamental form are
\[
\nu_t=N_t^{-1}(\partial_r-\beta_t),\qquad
\mathrm{II}_t=\frac1{2N_t}(\partial_r\gamma_r-\Lie_{\beta_t}\gamma_r).
\]
Write $H_t=\tr_{\gamma_r}\mathrm{II}_t$ and $H_0=m\cot r$. Then
\begin{align}
H_t&=\frac{H_0-td}{1+tu}
=H_0-tw+t^2uw+O(t^3),\label{eq:Ht}\\
\mathrm{II}_t^\circ&=-\frac{t}{2(1+tu)}(\Lie_v\gamma_r)^\circ.\notag
\end{align}

The scalar-curvature decomposition is
\begin{equation}\label{eq:gauss-lapse}
R(g_t)=R(\gamma_r)-\abs{\mathrm{II}_t}^2-H_t^2
-2N_t^{-1}(\partial_r-\beta_t)H_t
-2N_t^{-1}\Delta_{\gamma_r}N_t.
\end{equation}
The volume form is $N_t\,d\mu_{\gamma_r}\,dr$. After multiplication by
$\lambda$, the last term of \eqref{eq:gauss-lapse} integrates to zero on
each sphere. In the normal derivative term, integration by parts in $r$
uses $\partial_r d\mu_{\gamma_r}=H_0d\mu_{\gamma_r}$, while integration
by parts in the term involving $\beta_t$ uses
\[
H_0-\divg_{\gamma_r}\beta_t=N_tH_t.
\]
These two integrations produce $2\lambda N_tH_t^2+2\lambda'H_t$ in the
interior. Combining this with the original $-\lambda N_tH_t^2$ and using
$\abs{\mathrm{II}_t}^2=\abs{\mathrm{II}_t^\circ}^2+H_t^2/m$ gives, up to terms independent
of $t$,
\begin{equation}\label{eq:reduced-action}
I(g_t)=\int_0^\delta\int_{\Sph^m}
\left[\lambda N_t\left(R(\gamma_r)-n(n-1)
+\frac{m-1}{m}H_t^2-\abs{\mathrm{II}_t^\circ}^2\right)
+2\lambda'H_t\right]\dd\mu_{\gamma_r}\dd r.
\end{equation}
The boundary term $-2\lambda(\delta)\int_\Sigma H_t\,d\sigma$
cancels the variable part of the boundary term in $I$.

We justify the inner endpoint before differentiating. Smoothness of $h$
implies that $u$ and $\abs v_{\gamma_r}$ are bounded near the center,
and their first derivatives in polar coordinates are $O(r^{-1})$.
For small $t$, \eqref{eq:Ht} therefore gives $H_t=O(r^{-1})$;
the same bound holds for its first two $t$ derivatives. The boundary term at $r=\varepsilon$ and its first two variations are $O(\varepsilon^{m-1})$, which tend to zero
because $m\geq2$. The quadratic interior expressions have at worst size
$r^{-2}$ and are integrable against $r^m dr$.
The path \eqref{eq:ADM-path} gives the desired second variation even though
its second derivative in $t$ need not be smooth at the center. Indeed,
\[
\left.\partial_t^2g_t\right|_0
=2\bigl(u^2+\abs v_{\gamma_r}^2\bigr)\,dr^2=O(1),
\qquad \nabla\left(\left.\partial_t^2g_t\right|_0\right)=O(r^{-1}).
\]
Applying the integration-by-parts formula to this second derivative
produces a boundary term of order $O(\varepsilon^{m-1})$ at
$r=\varepsilon$. It tends to zero. Thus the second derivative of the path
does not contribute in the limit, and the second variation computed on
the annuli tends to $D^2I(\gb)[h,h]$.

The term $\lambda N_t(R(\gamma_r)-n(n-1))$ is affine in $t$. The other
second derivatives at zero are
\[
\left.\partial_t^2(N_tH_t^2)\right|_0=2w^2,\qquad
\left.\partial_t^2(N_t\abs{\mathrm{II}_t^\circ}^2)\right|_0
=\frac12\abs{(\Lie_v\gamma_r)^\circ}^2,\qquad
\left.\partial_t^2H_t\right|_0=2uw.
\]
Thus \eqref{eq:reduced-action} yields
\[
D^2I(\gb)[h,h]=\int_\Omega\left[
-\frac\lambda2\abs{(\Lie_v\gamma_r)^\circ}^2
+2\frac{m-1}{m}\lambda w^2+4\lambda'uw\right]\dd v.
\]
Multiplying by $-1/2$ proves \eqref{eq:radial-hessian}.
\end{proof}

\subsection{Tensor and vector modes}
\begin{lemma}\label{lem:vector-tensor}
A smooth tensor mode with $h^T=0$ satisfies
\begin{equation}\label{eq:tensor-positive}
Q(h)=\frac14\int_\Omega\lambda
(\abs{\nabla h}^2+2\abs h^2)\dd v,
\end{equation}
which is positive unless $h=0$. A smooth critical vector mode of the form
\eqref{eq:lapse-shift} satisfies
\begin{equation}\label{eq:vector-positive}
Q(h)=\frac14\int_\Omega\lambda
\abs{(\Lie_v\gamma_r)^\circ}^2\dd v\geq0.
\end{equation}
Equality in vector type holds exactly for boundary-fixing gauge tensors.
\end{lemma}
\begin{proof}
For a tensor mode, the identities
\[
\tr h=0,\qquad \divg h=0,\qquad h_{\nu\nu}=X_h=0
\]
give $\DR(h)=\DHg(h)=0$. Equations \eqref{eq:comparison} and
\eqref{eq:Q} give \eqref{eq:tensor-positive}.

For a vector mode of the form \eqref{eq:lapse-shift}, $u=0$ and $d=0$;
hence $w=0$.
Lemma~\ref{lem:radial-hessian} gives \eqref{eq:vector-positive}. Since $v$
is coexact, its Lie derivative has zero trace. Equality therefore requires
$v$ to be a Killing field on each concentric sphere. For one such harmonic,
write $v=b(r)V$. Choose
\[
\eta(r)=-\int_r^\delta\frac{b(\rho)}{\sin^2\rho}\dd\rho.
\]
Then $\eta(\delta)=0$, $s^2\eta'=b$, and
$\Lie_{\eta V^\sharp}\gb=2dr\odot bV$. By Lemma~\ref{lem:pole},
$b=r^3F(r^2)$, so $\eta=\eta(0)+r^2G(r^2)$. The generating vector field
is smooth and boundary-fixing. Adding back the tensor subtracted in
Lemma~\ref{lem:relative} proves the equality assertion for the original tensor.
\end{proof}

\subsection{The scalar factorization}
Fix a scalar harmonic of degree $\ell\geq2$, normalized by
$\int_{\Sph^m}Y^2\,d\mu_\gamma=1$, and use the representative
\eqref{eq:lapse-shift}. Write $u=u(r)Y$, $d=d(r)Y$, and
$w=w(r)Y$, and set
\[
\mu=\ell(\ell+m-1),\qquad L=\mu-m=(\ell-1)(\ell+m)>0.
\]
The critical constraints \eqref{eq:DH-radial}--\eqref{eq:DR-radial} become
\begin{equation}\label{eq:reconstruct}
u=-\frac{s^2}{L}\bigl(w'+(m+1)\cot r\,w\bigr),\qquad
d=w-m\cot r\,u,\qquad w(\delta)=0.
\end{equation}
Thus one radial function determines the scalar representative. Define
\[
\theta=\lambda s^m w.
\]

\begin{proposition}\label{prop:scalar}
For a smooth critical scalar representative \eqref{eq:lapse-shift}
of degree $\ell\geq2$,
\begin{equation}\label{eq:factor}
Q(h)=\frac{m(m-1)}{L\mu}\int_0^\delta\lambda s^{2-m}(\theta')^2\dd r
+\frac{m(m-1)}L\int_0^\delta\lambda s^{-m}\theta^2\dd r.
\end{equation}
In particular, $Q(h)\geq0$, with equality if and only if $h=0$.
\end{proposition}
\begin{proof}
Since $v$ is exact, write $v=b(r)DY$, so $d=-\mu b/s^2$. The sphere Bochner formula gives
\begin{align*}
\int_{\Sph^m}\abs{D^2Y}^2\dd\mu_\gamma
&=\mu^2-(m-1)\mu,\\
\int_{\Sph^m}\abs{Y^\circ}^2\dd\mu_\gamma
&=\frac{m-1}{m}\mu(\mu-m).
\end{align*}
As $\Lie_v\gamma_r=2bD^2Y$, after raising indices with $\gamma_r$ this
becomes
\[
\frac14\int_{(\Sph^m,\gamma_r)}
\abs{(\Lie_v\gamma_r)^\circ}^2\dd\mu_{\gamma_r}
=\frac{m-1}{m}\frac L\mu
\int_{(\Sph^m,\gamma_r)}d^2\dd\mu_{\gamma_r}.
\]
Consequently \eqref{eq:radial-hessian} reads
\[
Q(h)=\int_0^\delta s^m
\left[\frac{m-1}{m}\lambda\left(\frac L\mu d^2-w^2\right)
-2\lambda'uw\right]\dd r.
\]

To isolate the positive terms and the endpoint contribution, substitute
\eqref{eq:reconstruct} and $\theta=\lambda s^m w$. The resulting pointwise
identity is
\begin{align}\label{eq:pointwise-factor}
s^m\left[\frac{m-1}{m}\lambda
\left(\frac L\mu d^2-w^2\right)-2\lambda'uw\right]
={}&\frac{m(m-1)}{L\mu}\lambda s^{2-m}(\theta')^2
+\frac{m(m-1)}L\lambda s^{-m}\theta^2\\
&+\frac{d}{dr}\left[
\frac{s^{m+1}(m\lambda^2-1)}Lw^2\right].\notag
\end{align}
We verify this identity by comparing coefficients of $w'$ and $w$.
Substitution gives
\[
u=-\frac{s^2}Lw'-\frac{(m+1)s\lambda}Lw,
\qquad
d=\frac{m s\lambda}Lw'
+\frac{L+m(m+1)\lambda^2}Lw,
\]
and
\[
\theta'=\lambda s^m\left[w'+(m\cot r-\tan r)w\right].
\]
Both sides of \eqref{eq:pointwise-factor} have the same $(w')^2$
coefficient,
\[
\frac{m(m-1)}{L\mu}\lambda^3s^{m+2}.
\]
After this coefficient is removed, the difference between the left side
and the two positive terms is
\[
\frac{2s^{m+1}(m\lambda^2-1)}Lww'
+\frac{\lambda s^m}L
\left[(m+1)(m\lambda^2-1)-2ms^2\right]w^2.
\]
The second coefficient is the derivative of
$s^{m+1}(m\lambda^2-1)/L$, since $s'=\lambda$ and $\lambda'=-s$.
This proves \eqref{eq:pointwise-factor}.

It remains to justify the endpoints. At $r=\delta$, the total derivative
contributes zero because $w(\delta)=0$. Near the center, smoothness of
$h$ gives $u=O(1)$ and $|v|_{\gamma_r}=O(1)$, hence $b=O(r)$.
The definitions of $d,w$ and \eqref{eq:reconstruct} imply
\[
d,w=O(r^{-1}),\qquad w'=O(r^{-2}),\qquad
\theta=O(r^{m-1}),\qquad \theta'=O(r^{m-2}).
\]
Consequently,
\[
s^{m+1}(m\lambda^2-1)w^2=O(r^{m-1})\longrightarrow0,
\qquad
\lambda s^{2-m}(\theta')^2+\lambda s^{-m}\theta^2=O(r^{m-2}).
\]
Since $m\geq2$, the endpoint term vanishes and both positive integrals
are finite. Integrating \eqref{eq:pointwise-factor} on an annulus and
passing to the center proves \eqref{eq:factor}.

All coefficients in \eqref{eq:factor} are strictly positive in the open
interval, and the second integral forces $\theta=0$ in an equality case. Hence
$w=0$, and \eqref{eq:reconstruct} gives $u=d=0$. Since the one-form
coefficient satisfies $d=-\mu b/s^2$, it also vanishes.
\end{proof}

\subsection{Low scalar degrees and the equality cases}
Formula \eqref{eq:reconstruct} requires $L\ne0$. We treat scalar degrees
zero and one by solving their constraints directly.

\begin{lemma}\label{lem:low-scalar}
Every smooth critical scalar mode of degree $0$ or $1$ has the form
\[
h=\Lie_X\gb,\qquad X|_\Sigma=0,
\]
with $X$ smooth. In degree zero, the representative
\eqref{eq:lapse-shift} vanishes.
\end{lemma}
\begin{proof}
Subtract the gauge tensor from Lemma~\ref{lem:relative} and use
\eqref{eq:lapse-shift}. In degree $\ell=0$, there is no mixed component: $v=d=0$ and
$w=m\cot r\,u$. Equation \eqref{eq:DR-radial} reduces to
\[
(\lambda s^m w)'=0,\qquad w(\delta)=0.
\]
Thus $w=u=0$.

In degree $\ell=1$, $\mu=m$, so the two terms involving $u$ in
\eqref{eq:DR-radial} cancel. Hence
\[
(s^{m+1}w)'=0,\qquad w(\delta)=0,
\]
and therefore $w=0$. For $v=b(r)DY$, the relation
$d=-mb/s^2$ and $w=m\cot r\,u+d=0$ yield $b=s\lambda u$.
Choose $\xi$ by
\begin{equation}\label{eq:degree-one-gauge}
\xi'=u,\qquad \xi(\delta)=0,
\qquad X=\xi Y\partial_r+\cot r\,\xi\nabla_\gamma Y.
\end{equation}
Since $D^2Y=-Y\gamma$ in degree one,
\[
(\Lie_X\gb)_{rr}=2uY,\qquad (\Lie_X\gb)_{ab}=0.
\]
Its mixed component is
\[
(\Lie_X\gb)_{ra}
=\left[\xi+s^2(\cot r\,\xi)'\right]D_aY
=s\lambda uD_aY,
\]
where $s^2(\cot r)'=-1$ cancels the undifferentiated $\xi$ term.
Thus $h=\Lie_X\gb$.

The field in \eqref{eq:degree-one-gauge} is boundary-fixing. For smoothness
at the center, degree-one tensor parity gives
$u=rF(r^2)$ and hence $\xi=\xi(0)+r^2G(r^2)$. With
$\eta=\cot r\,\xi$, one has
\[
\eta=r^{-1}F_1(r^2),\qquad
\xi-r\eta=\xi(1-r\cot r)=r^2G_1(r^2).
\]
These are exactly \eqref{eq:pole-vector} for $\ell=1$. The generating
field is therefore smooth through the center. Adding back the gauge
tensor from Lemma~\ref{lem:relative} proves the assertion for the
original degree-one mode.
\end{proof}

\begin{proposition}\label{prop:smooth-positive}
For a smooth critical tensor in a single spherical block,
\[
Q(h)\geq0,\qquad Q(h)=0
\quad\Longleftrightarrow\quad
h=\Lie_X\gb\text{ for some smooth }X\text{ with }X|_\Sigma=0.
\]
\end{proposition}
\begin{proof}
Decompose the block into its finitely many orthonormal angular components
and apply Lemma~\ref{lem:relative} to each scalar or vector component.
Lemma~\ref{lem:pole} ensures smoothness of the gauge fields, and
Lemma~\ref{lem:comparison} preserves $Q$ under these changes.
For the resulting tensors, the angular orthogonality identities proved
in Lemma~\ref{lem:blocks}, together with the radial formulas in
Lemmas~\ref{lem:radial-hessian} and \ref{lem:vector-tensor}, give
\[
h=\sum_jh_j,\qquad Q(h)=\sum_jQ(h_j).
\]
Proposition~\ref{prop:scalar} and Lemmas~\ref{lem:vector-tensor} and
\ref{lem:low-scalar} give $Q(h_j)\geq0$, with equality precisely for
gauge tensors. Hence
\[
Q(h)=0\quad\Longrightarrow\quad
h_j=\Lie_{X_j}\gb\quad\Longrightarrow\quad
h=\Lie_{\sum_jX_j}\gb,\qquad X_j|_\Sigma=0.
\]
The converse follows from Lemma~\ref{lem:comparison}.
\end{proof}

\section{Positivity on the weak constraint space}\label{sec:weak}
Smooth approximation in each spherical mode extends the positivity in
Section~\ref{sec:geometry} to weakly constrained tensors.

\subsection{Weak constraints and compactness}
The weak joint condition can be defined without a normal derivative of an
$H^1$ tensor. For $h\in H^1$ with $h^T=0$, it is
\begin{equation}\label{eq:weak-joint}
\int_\Omega\ip{\DRs \varphi}{h}\dd v=0
\qquad\text{for every }\varphi\in C^\infty(\overline\Omega).
\end{equation}
For smooth $h$, compactly supported tests first imply $\DR(h)=0$;
\eqref{eq:green-fixed} and arbitrary boundary traces then imply $\DHg(h)=0$.
The weak critical space in divergence gauge is therefore
\[
Z=\left\{h\in H^1(\Omega;S^2T^*\Omega):
h^T=0,\ \divg h=0,\ h\text{ satisfies \eqref{eq:weak-joint}}\right\}.
\]
It is a closed subspace of $H^1$.

We use the compact embeddings $H^1(\Omega)\hookrightarrow L^2(\Omega)$ and
$H^{1/2}(\Sigma)\hookrightarrow L^2(\Sigma)$, together with the continuous
trace map $H^1(\Omega)\to H^{1/2}(\Sigma)$. Consequently a weakly convergent
$H^1$ sequence has, after passage to a subsequence, strongly convergent
volume and boundary $L^2$ traces. For divergence-free tensors with
$h^T=0$, the derivative terms in \eqref{eq:Q} are
nonnegative convex quadratic forms; all remaining terms are continuous
under these strong convergences. Hence
\[
h_j\rightharpoonup h\text{ in }H^1,\quad \divg h_j=0,\quad h_j^T=0
\quad\Longrightarrow\quad Q(h)\leq\liminf_{j\to\infty}Q(h_j).
\]

\subsection{Spherical decomposition and the constraints}
For a block projection from Section~\ref{subsec:modes}, put
\[
Z_\beta=Z\cap\ran\Pi_\beta.
\]
The associated bilinear form is denoted by $Q(h,k)$, with
$Q(h,k)=\tfrac12[Q(h+k)-Q(h)-Q(k)]$.

\begin{lemma}\label{lem:blocks}
The projections $\Pi_\beta$ are bounded on $H^1$, commute with boundary
restriction, and preserve $Z$. For every increasing exhaustion $F_N$ of
the block indices,
\begin{equation}\label{eq:angular-convergence}
\sum_{\beta\in F_N}\Pi_\beta h\longrightarrow h\quad\hbox{in }H^1.
\end{equation}
Distinct blocks and distinct orthonormal basis elements within a block
are $L^2$-orthogonal. For divergence-free $h,k$ with $h^T=k^T=0$, they
are also $Q$-orthogonal:
\begin{equation}\label{eq:orthogonal}
\int_\Omega\ip{\Pi_\beta h}{\Pi_{\beta'}k}\dd v=0,
\qquad Q(\Pi_\beta h,\Pi_{\beta'}k)=0\quad(\beta\ne\beta').
\end{equation}
Consequently,
\begin{equation}\label{eq:sum-Q}
Q(h)=\lim_{N\to\infty}\sum_{\beta\in F_N}Q(\Pi_\beta h)
\qquad(h\in H^1,\ \divg h=0,\ h^T=0).
\end{equation}
\end{lemma}
\begin{proof}
\textit{Step 1: Angular orthogonality.}
Write $S(V)_{ab}=D_aV_b+D_bV_a$. Integration by parts on the unit sphere gives
\[
\int_{\Sph^m}\ip{DY}{V}\dd\mu_\gamma
=\int_{\Sph^m}\ip{Y^\circ}{S(V)}\dd\mu_\gamma=0,
\]
and, for a trace-free divergence-free tensor harmonic $T$,
\[
\int_{\Sph^m}\ip{T}{Y\gamma}\dd\mu_\gamma
=\int_{\Sph^m}\ip{T}{Y^\circ}\dd\mu_\gamma
=\int_{\Sph^m}\ip{T}{S(V)}\dd\mu_\gamma=0.
\]
The corresponding derivative pairings follow from
\begin{align*}
D^aY^\circ_{ab}&=-\frac{m-1}{m}(\mu-m)D_bY,
&D^aS(V)_{ab}&=-(\mu-m)V_b,\\
-D^cD_cY^\circ&=(\mu-2m)Y^\circ,
&-D^cD_cS(V)&=(\mu-m-2)S(V).
\end{align*}
Different degrees are orthogonal by
self-adjointness of the spherical Laplacians. Within each degree,
integration by parts reduces the pairings to the inner products of the
chosen orthonormal basis elements.

Take $e_0=\partial_r$ and $e_a=s^{-1}\widetilde e_a$, where
$\widetilde e_a$ is a $\gamma$-orthonormal frame. Then
\[
(\nabla_{e_0}h)_{ij}=h_{ij}',\qquad
(\nabla_{e_c}h)_{ij}=s^{-1}D_ch_{ij}
+\cot r\left(\delta_{ic}h_{0j}+\delta_{jc}h_{i0}
-\delta_{i0}h_{cj}-\delta_{j0}h_{ic}\right).
\]
Integration by parts reduces the mixed terms in the gradient pairing to
the trace and divergence pairings above. The terms with two angular
derivatives use the displayed Laplacian identities; radial derivatives
preserve each mode.
Consequently, on every annulus $A_{a,b}=\{a<r<b\}$,
\[
\int_{A_{a,b}}\ip{\nabla\Pi_\beta h}{\nabla\Pi_{\beta'}k}\dd v=0
\qquad(\beta\ne\beta').
\]
The same holds for the tensors, their traces, and the gradients of their
traces, including distinct orthonormal basis elements within one block.
Thus, for every finite set $F$, the sum $P_F=\sum_{\beta\in F}\Pi_\beta$ satisfies
\begin{equation}\label{eq:projection-bound}
\norm{P_Fh}_{H^1(A_{a,b})}^2
+\norm{h-P_Fh}_{H^1(A_{a,b})}^2
=\norm h_{H^1(A_{a,b})}^2.
\end{equation}

\textit{Step 2: $H^1$ convergence at the center.}
First let $h$ be smooth. Equation~\eqref{eq:projection-bound} bounds
$\norm{P_Fh}_{H^1(A_{a,\delta})}$ independently of $a>0$.
These are also weak derivatives across the center. To see this directly,
let $u$ be any component of $P_Fh$ in smooth local coordinates and let
$\zeta_\varepsilon$ vanish on $B_\varepsilon$ and equal one outside
$B_{2\varepsilon}$, with $|\nabla\zeta_\varepsilon|\leq C/\varepsilon$.
In integration by parts against a smooth test function $\varphi$, the
additional term is bounded by
\[
\left|\int u\varphi\,\partial_i\zeta_\varepsilon\dd v\right|
\leq C\norm\varphi_{L^\infty}
\norm u_{L^2(B_{2\varepsilon}\setminus B_\varepsilon)}
\varepsilon^{(n-2)/2}\longrightarrow0.
\]
Hence $P_Fh\in H^1(\Omega)$, and \eqref{eq:projection-bound} holds also
on $B_b$ after letting $a\downarrow0$.

Put $P_N=P_{F_N}$. Smooth harmonic expansions converge with their first
derivatives on each closed annulus. For any $0<\rho<\delta$, therefore,
\begin{align*}
\norm{h-P_Nh}_{H^1(\Omega)}^2
&\leq\norm{h-P_Nh}_{H^1(\Omega\setminus B_\rho)}^2
+\norm h_{H^1(B_\rho)}^2,\\
\limsup_{N\to\infty}\norm{h-P_Nh}_{H^1(\Omega)}^2
&\leq\norm h_{H^1(B_\rho)}^2\longrightarrow0
\qquad(\rho\downarrow0).
\end{align*}
The bound $\norm{P_Fh}_{H^1}\leq\norm h_{H^1}$ extends $P_F$ to all $H^1$ tensors. If smooth
$h_j\to h$ in $H^1$, then
\[
\norm{h-P_Nh}_{H^1}
\leq2\norm{h-h_j}_{H^1}+\norm{h_j-P_Nh_j}_{H^1}.
\]
Letting first $N\to\infty$ and then $j\to\infty$ proves
\eqref{eq:angular-convergence}.

\textit{Step 3: Constraints and boundary terms.}
Restriction to $r=\delta$ commutes with angular projection, so $h^T=0$
is preserved. The polar-coordinate divergence formulas also give
\[
\divg(\Pi_\beta h)=\widehat\Pi_\beta(\divg h),
\]
where $\widehat\Pi_\beta$ denotes the corresponding one-form projection:
formed from $Y\,dr$ and $dY$ in scalar type, from divergence-free
one-form harmonics in vector type, and zero in tensor type. The identity extends from smooth to $H^1$ tensors by continuity.
In particular, divergence-free tensors remain divergence-free.

For the weak joint condition, let $\pi_\ell$ be the scalar harmonic
projection. The formula for the Hessian in polar coordinates gives
\[
\Pi_{\ell,\mathrm{sc}}\DRs \varphi=\DRs(\pi_\ell \varphi),\qquad
\Pi_{\ell,\mathrm{vec}}\DRs \varphi
=\Pi_{\ell,\mathrm{TT}}\DRs \varphi=0.
\]
If $\varphi$ is smooth, then $\pi_\ell \varphi$ is smooth also at the center.
Indeed, in normal coordinates the homogeneous Taylor polynomial of
degree $j$ is a sum of terms $|x|^{2i}H_{j-2i}(x)$, where $H_{j-2i}$
is homogeneous and harmonic. The projection retains just those terms
with $j-2i=\ell$, which are polynomials. The remainder $R_N$ after Taylor expansion through order $N$ satisfies
\[
|\partial^k\pi_\ell R_N(x)|\leq C_{N,k,\ell}|x|^{N+1-k}
\qquad(k\leq N).
\]
Radial derivatives commute with the projection, and angular derivatives
are bounded on its fixed finite-dimensional eigenspace. Expressing
derivatives in polar coordinates introduces at most one factor $r^{-1}$
per derivative. As $N$ is arbitrary,
this proves smoothness at the center.
Thus, for $h\in Z$ and every smooth $\varphi$, self-adjointness of the projections gives
\[
\int_\Omega\ip{\DRs \varphi}{\Pi_{\ell,\mathrm{sc}}h}\dd v
=\int_\Omega\ip{\DRs(\pi_\ell \varphi)}h\dd v=0;
\]
the vector and tensor pairings vanish directly. This proves
$\Pi_\beta h\in Z$.

Finally, the normal--normal boundary component is scalar, while the
mixed boundary one-form is a multiple of $dY$ in scalar type,
divergence-free in vector type, and zero in tensor type. Hence the boundary pairings in \eqref{eq:Q} are
orthogonal as well. Together with Step~1 and the radial nature of
$\lambda$, this proves \eqref{eq:orthogonal}. Continuity of $Q$ on the divergence-free $H^1$ space and \eqref{eq:angular-convergence} now give
\[
Q(h)=\lim_{N\to\infty}Q(P_Nh)
=\lim_{N\to\infty}\sum_{\beta\in F_N}Q(\Pi_\beta h),
\]
which is \eqref{eq:sum-Q}.
\end{proof}

\subsection{Approximation in each spherical mode}

\begin{lemma}\label{lem:mode-density}
Every $h\in Z_\beta$ admits smooth $h_j\in Z_\beta$ with
\[
h_j\to h\text{ in }H^1,\qquad \tr h_j=\tr h,\qquad
h_j|_\Sigma=h|_\Sigma.
\]
\end{lemma}
\begin{proof}
It suffices to consider one angular basis element. For $h\in Z_\beta$,
$\tau=\tr h$ satisfies $(\Delta+m)\tau=0$, where $m=n-1$.
Interior elliptic regularity gives smoothness at the center. In scalar
type, write $\tau=\vartheta(r)Y$; its radial equation is
\[
\vartheta''+m\cot r\,\vartheta'
+\left(m-\frac\mu{s^2}\right)\vartheta=0.
\]
The coefficients are smooth near $r=\delta$, so this solution extends
smoothly to the boundary. In vector and tensor types $\tau=0$.
Thus $\tau\in C^\infty(\overline\Omega)$, and
\[
h_0=-\frac1m\Hess\tau,\qquad \widetilde h=h-h_0,
\qquad \tr h_0=\tau,\quad\divg h_0=0,\quad\tr \widetilde h=\divg \widetilde h=0.
\]
We approximate $\widetilde h$ by smooth trace-free divergence-free tensors with the
same boundary values of all components. Such approximations $\widetilde h_j$ give $h_j=h_0+\widetilde h_j\in Z_\beta$, since
\[
\int_\Omega\ip{\DRs\varphi}{h_j-h}\dd v
=\int_\Omega\ip{\Hess\varphi}{\widetilde h_j-\widetilde h}\dd v=0
\]
by integration by parts.

For scalar degree $\ell\geq2$, put
$A=s^2/\mu$ and $\kappa_\ell=(m-1)(\mu-m)/m>0$.
The trace and divergence equations for the coefficients of $\widetilde h$ are
\[
p=-\frac am,\qquad
b=A\bigl(a'+(m+1)\cot r\,a\bigr),\qquad
f=\frac{b'+m\cot r\,b-a/m}{\kappa_\ell}.
\]
On closed annuli these identities give $a\in H^3$ and $b\in H^2$.
Let $\zeta_\varepsilon$ be zero for $r\leq\varepsilon$ and one for
$r\geq2\varepsilon$, with $|\zeta_\varepsilon^{(j)}|\leq C_j\varepsilon^{-j}$
for $j\leq3$. Suppress its subscript and reconstruct $\widetilde h_\zeta$ from $a_\zeta=\zeta a$:
\[
p_\zeta=-\frac{\zeta a}{m},\qquad
b_\zeta=\zeta b+A\zeta'a,\qquad
f_\zeta=\zeta f+\frac{2\zeta'b+A(\zeta''+\cot r\,\zeta')a}{\kappa_\ell}.
\]
Then $\tr \widetilde h_\zeta=\divg \widetilde h_\zeta=0$. Using
\[
a'=b/A-(m+1)\cot r\,a,\qquad b'=\kappa_\ell f-m\cot r\,b+a/m,
\]
the reconstruction formulas give, for $E_\zeta=\widetilde h_\zeta-\zeta \widetilde h$,
\[
\int_{\Sph^m}\left(|E_\zeta|^2+r^2|\nabla E_\zeta|^2\right)\dd\mu_\gamma
\leq C_\beta\int_{\Sph^m}
\left(|\widetilde h|^2+r^2|\nabla \widetilde h|^2\right)\dd\mu_\gamma.
\]
Here $E_\zeta$ is supported in $\varepsilon<r<2\varepsilon$;
the tensor norms are taken with respect to $\gb$ at radius $r$.
The two displayed equations control $a'$ and $b'$, while the derivative
bounds for $\zeta$ control the remaining coefficients on this annulus.
Integrating in $r$ and estimating $(\zeta-1)\widetilde h$ gives
\[
\norm{\widetilde h_\zeta-\widetilde h}_{H^1}^2
\leq C_\beta\int_{r<2\varepsilon}
\left(|\nabla \widetilde h|^2+\frac{|\widetilde h|^2}{r^2}+|\widetilde h|^2\right)\dd v
\longrightarrow0.
\]
The local Hardy inequality gives $\widetilde h/r\in L^2$ because $n\geq3$.
Approximate $a_\zeta$ in $H^3$ by smooth functions vanishing near the
center and preserving $a_\zeta,a_\zeta',a_\zeta''$ at $\delta$.
The reconstruction gives $H^1$ convergence and preserves the boundary values of all tensor components.

For a non-Killing vector mode, put $L=\mu-m>0$. The divergence equation is
\[
b'+m\cot r\,b=Lf,
\]
so $b\in H^2$ on closed annuli. The coefficients
\[
b_\zeta=\zeta b,\qquad f_\zeta=\zeta f+\frac{\zeta'b}{L}
\]
give the same cutoff estimate. Smooth $H^2$ approximation of $b_\zeta$
preserving its value and first derivative at $\delta$ completes this case.
For tensor modes, an ordinary cutoff and $H^1$ approximation of $f$
with its boundary value fixed preserve both constraints.

The exceptional modes have no nonzero $L^2$ trace-free divergence-free
tensor. In scalar degree zero, the divergence equation gives
$a=C s^{-m-1}$. In scalar degree one, setting $B=s^m b$ gives
\[
B''+\cot r\,B'-s^{-2}B=0,\qquad
B=C_1\tan(r/2)+C_2\cot(r/2),\qquad a=m s^{-m}B'.
\]
The two branches of $a$ have orders $r^{-m}$ and $r^{-m-2}$ at the center;
neither is in $L^2$. A Killing vector mode has
$b=C s^{-m}$, whose mixed tensor is likewise not in $L^2$ unless $C=0$.
Thus $\widetilde h=0$ in these cases and $h=h_0$ is already smooth.
A diagonal choice of the cutoffs and radial approximations proves the lemma.
\end{proof}

\subsection{Positivity on the weak space}
\begin{theorem}\label{thm:weak-positive}
On every proper cap,
\[
Q(h)\geq0\quad(h\in Z),\qquad Q(h)=0\Longleftrightarrow h=0.
\]
\end{theorem}
\begin{proof}
For smooth tensors in $Z_\beta$, Lemma~\ref{lem:comparison} and
Proposition~\ref{prop:smooth-positive} give $Q\geq0$.
Lemma~\ref{lem:mode-density} and continuity of $Q$ extend nonnegativity
to $Z_\beta$. The orthogonal decomposition then gives $Q\geq0$ on $Z$.
If $Q(h)=0$, every block has zero quadratic form. By angular orthogonality,
it suffices to consider one normalized angular basis element in a block.

Choose smooth $h_j\in Z_\beta$ converging to $h$ in $H^1$.
Then $Q(h_j)\to0$. In scalar degrees $\ell\geq2$ and non-Killing vector
modes, let $X_j$ be the gauge fields from Lemma~\ref{lem:relative}.
Their coefficient formulas give
\[
\norm{X_j-X_i}_{H^1}\leq C_\beta\norm{h_j-h_i}_{H^1}.
\]
Indeed, in a radially parallel orthonormal frame their coefficients are
$r$ times smooth bounded functions times the coefficients of $h_j$.
A radial derivative involves only these coefficients and $r$ times their
first derivatives. Angular derivatives are controlled in the fixed mode.
Thus $X_j\to X$ in $H^1$, with $X|_\Sigma=0$, and
\[
h_j-\Lie_{X_j}\gb\longrightarrow h-\Lie_X\gb\quad\hbox{in }L^2.
\]
In scalar type, \eqref{eq:factor} gives $\theta_j\to0$ in $H^1$ on every
closed annulus. The reconstruction \eqref{eq:reconstruct} then gives
$u_j,d_j,b_j\to0$ in $L^2$ there. In a non-Killing vector mode normalized
by $\int_{\Sph^m}|V|^2=1$, the remaining mixed coefficient satisfies
\[
Q(h_j)=\frac{\mu-m}{2}\int_0^\delta\lambda s^{m-4}b_j^2\dd r,
\]
so it also tends to zero on each closed annulus. In either case
$h=\Lie_X\gb$ almost everywhere. Since $X|_\Sigma=0$ and $\divg h=0$,
\[
\norm h_{L^2}^2
=\int_\Omega\ip h{\Lie_X\gb}\dd v
=-2\int_\Omega\ip{\divg h}{X}\dd v=0.
\]
In tensor type, \eqref{eq:tensor-positive} extends by $H^1$ convergence
and immediately gives $h=0$. The low scalar modes are smooth by the
preceding lemma, so their equality cases follow from
Proposition~\ref{prop:smooth-positive} and the same integration by parts.
The Killing vector block is zero. Hence every block vanishes, and $h=0$.
\end{proof}

\section{Proof of local rigidity}\label{sec:nonlinear}
The positive static potential forces every normalized weak limit of
admissible perturbations to satisfy the joint linearized constraints.
Weak positivity and the G\aa rding inequality then complete the proof.

\subsection{A Taylor estimate}
\begin{lemma}\label{lem:tame}
If $h\in C^2$, $h^T=0$, and $\norm h_{C^2}$ is sufficiently small, then
every $\varphi\in C^\infty(\overline\Omega)$ satisfies
\begin{multline*}
\bigg|\int_\Omega \varphi\bigl(R(\gb+h)-n(n-1)\bigr)\dd v
+2\int_\Sigma \varphi\bigl(H(\gb+h)-H(\gb)\bigr)\dd\sigma\\
-\int_\Omega\ip{\DRs \varphi}{h}\dd v\bigg|
\leq C_\varphi\norm h_{C^2}\norm h_{H^1}.
\end{multline*}
No divergence condition is required.
\end{lemma}
\begin{proof}
Put
\[
E_R=R(\gb+h)-n(n-1)-\DR(h),\qquad
E_H=H(\gb+h)-H(\gb)-\DHg(h).
\]
The coordinate formulas for scalar and mean curvature give the Taylor
bounds
\[
|E_R|\leq C\bigl(|h|\,|\nabla^2h|+|\nabla h|^2+|h|^2\bigr),
\qquad
|E_H|\leq C\bigl(|h|\,|\nabla h|+|h|^2\bigr).
\]
Indeed, scalar curvature is affine in the second derivatives of the
metric, and mean curvature is affine in its first derivatives. Their
coefficients depend smoothly on the metric. Taking one factor in
$L^\infty$ and using the trace theorem gives
\[
\norm{E_R}_{L^2(\Omega)}+\norm{E_H}_{L^2(\Sigma)}
\leq C\norm h_{C^2}\norm h_{H^1}.
\]
The assertion follows by Cauchy--Schwarz and the integration-by-parts identity
\eqref{eq:green-fixed}.
\end{proof}

\subsection{The compactness argument}
\begin{proof}[Proof of Theorem~\ref{thm:main}]
Suppose the theorem fails. After applying the boundary-fixing diffeomorphism
from Lemma~\ref{lem:slice}, there are nonzero admissible perturbations
\[
g_j=\gb+h_j,\qquad \divg h_j=0,\qquad h_j^T=0,
\qquad \norm{h_j}_{C^{2,\alpha}}\longrightarrow0.
\]
Set
\[
\varepsilon_j=\norm{h_j}_{H^1},\qquad v_j=h_j/\varepsilon_j,
\qquad \norm{v_j}_{H^1}=1.
\]
Passing to a subsequence, we have
\[
v_j\rightharpoonup v\text{ in }H^1,\qquad
v_j\longrightarrow v\text{ in }L^2(\Omega)\text{ and }L^2(\Sigma),
\qquad \divg v=0,\quad v^T=0.
\]

For a smooth test function $\varphi$, write
\[
E_\varphi(g)=\int_\Omega \varphi\bigl(R(g)-n(n-1)\bigr)\dd v
+2\int_\Sigma \varphi\bigl(H(g)-H(\gb)\bigr)\dd\sigma.
\]
Since $\DRs\lambda=0$, Lemma~\ref{lem:tame} and admissibility give
\[
0\leq E_\lambda(g_j)
\leq C\norm{h_j}_{C^2}\varepsilon_j.
\]
Both curvature defects are nonnegative and $\lambda\geq c>0$. Thus
\[
|E_\varphi(g_j)|\leq\frac{\norm \varphi_{L^\infty}}c E_\lambda(g_j),
\qquad
\left|\int_\Omega\ip{\DRs \varphi}{v_j}\dd v\right|
\leq C_\varphi\norm{h_j}_{C^2}\longrightarrow0.
\]
The second inequality follows by another application of
Lemma~\ref{lem:tame}. Passing to the limit proves the weak joint
constraint \eqref{eq:weak-joint}, so $v\in Z$.

For large $j$, admissibility gives $D(g_j)\geq0$. Dividing the
Brendle--Marques expansion \eqref{eq:BM} by $\varepsilon_j^2$ yields
\[
Q(v_j)\leq C\norm{h_j}_{C^1}=o(1).
\]
Weak lower semicontinuity gives $Q(v)\leq0$, and
Theorem~\ref{thm:weak-positive} forces $v=0$. The G\aa rding inequality
now gives the contradiction
\[
0<\kappa
\leq Q(v_j)+C\norm{v_j}_{L^2}^2
\leq C\norm{h_j}_{C^1}+C\norm{v_j}_{L^2}^2
\longrightarrow0.
\]
\end{proof}

\appendix
\section{Comparison with the Pohozaev form}\label{app:pohozaev}
We prove \eqref{eq:intro-pohozaev} and give its finite-radius version.
The hemisphere is $\Omega_+=\{x\in\Sph^n:x_n\geq0\}$, and the static
function remains $\lambda=x_{n+1}$. On this domain, for variations with
$\DR(\widehat h)=\DHg(\widehat h)=0$, we use the same curvature
second-variation formula as in the main text:
\begin{equation}\label{eq:common-hessian}
Q(\widehat h)=-\frac12\int_{\Omega_+}\lambda
D^2R_{\gb}(\widehat h,\widehat h)\dd v
-\int_{\partial\Omega_+}\lambda
D^2H_{\gb}(\widehat h,\widehat h)\dd\sigma.
\end{equation}
On the cap this is \eqref{eq:critical-hessian}; here the induced
boundary metric is allowed to vary.

Use stereographic coordinates
\[
\Psi(y)=\left(\frac{2y}{1+|y|^2},\frac{1-|y|^2}{1+|y|^2}\right),
\qquad \Psi(\overline{\R^n_+})=\Omega_+\setminus\{-e_{n+1}\}.
\]
We identify spherical tensors with their pullbacks by $\Psi$.
In Euclidean formulas below, $\nabla$ and $\Delta$ denote the
Euclidean gradient and Laplacian. Set
\[
U=(1+|y|^2)^{-(n-2)/2},\qquad
\gb=4U^{4/(n-2)}g_{\mathrm{euc}},\qquad
\mathcal D=\frac{n-2}{2}+y\cdot\nabla.
\]
The dilation multiplier gives the same static function:
\[
\frac2{n-2}\frac{\mathcal DU}{U}
=\lambda\circ\Psi=\frac{1-|y|^2}{1+|y|^2},\qquad
\DRs\lambda=0,\qquad
\partial_\nu\lambda=0,\quad\mathrm{II}_{\gb}=0
\quad\text{on }\partial\Omega_+.
\]

Let $h$ be a smooth Euclidean trace-free symmetric tensor and let
$g_t=\exp(th)$ denote the exponential of its coefficient matrix. Define
\[
L_g=\Delta_g-\frac{n-2}{4(n-1)}R(g),\qquad
B_g=\partial_{\nu_g}+\frac{n-2}{2(n-1)}H(g),\qquad
\omega_t=\sqrt{\det((g_t)_{ij})_{i,j<n}}.
\]
Fix a linear choice of smooth conformal correction $z=z[h]$ satisfying
\begin{equation}\label{eq:yamabe-correction}
\begin{aligned}
(\Delta+n(n+2)U^{4/(n-2)})z
&=\left.\partial_t\bigl((\Delta-L_{g_t})U\bigr)\right|_{t=0}
&&\text{in }\R^n_+,\\
\partial_nz
&=\left.\partial_t\bigl(\omega_t B_{g_t}U\bigr)\right|_{t=0}
&&\text{on }\partial\R^n_+.
\end{aligned}
\end{equation}
Put $u_t=U+tz$. On each fixed half-ball and for small $|t|$,
\[
\widehat g_t=4u_t^{4/(n-2)}\exp(th)
=\gb+t\widehat h+t^2\widehat h^{(2)}+O(t^3),
\]
where $\widehat h$ is given in \eqref{eq:intro-pohozaev} and
\[
\widehat h^{(2)}=\frac12\left.\partial_t^2\widehat g_t\right|_{t=0}
=4U^{4/(n-2)}\left(
\frac12h^2+\frac{4z}{(n-2)U}h
+\frac{2(6-n)z^2}{(n-2)^2U^2}g_{\mathrm{euc}}\right).
\]
Here $h^2$ is formed using the Euclidean metric. Conformal covariance and
\eqref{eq:yamabe-correction} give
\[
\DR(\widehat h)=0,\qquad \DHg(\widehat h)=0.
\]

For $B_R^+=\{|y|<R,y_n>0\}$ and $D_R=\{|y'|<R,y_n=0\}$, define
\begin{align*}
\mathcal G_t^R(\varphi,\psi)
&=\int_{B_R^+}\varphi(\Delta-L_{g_t})\psi\dd y
+\int_{D_R}\varphi(\partial_n+\omega_tB_{g_t})\psi\dd y',\\
\mathcal P_R(h,h)&=[t^2]\mathcal G_t^R(\mathcal Du_t,u_t),
\qquad \mathcal P(h,h)=\lim_{R\to\infty}\mathcal P_R(h,h).
\end{align*}
The symbol $[t^2]$ denotes the quadratic coefficient, and the last
definition applies when the limit exists.

Let $\Omega_R$, $\Gamma_R$, and $S_R$ be the images under $\Psi$ of
$B_R^+$, $D_R$, and $S_R^+=\{|y|=R,y_n\geq0\}$, respectively, and
put $E_R=\partial\Gamma_R$. Set
\[
I_R(g)=\int_{\Omega_R}\lambda\bigl(R(g)-n(n-1)\bigr)\dd v_g
+2\int_{\Gamma_R}\lambda H(g)\dd\sigma_g,
\qquad Q_R(\widehat h)=-\frac12D^2I_R(\gb)[\widehat h,\widehat h].
\]
The mean-curvature integral is only over $\Gamma_R$. The remaining
boundary contribution of $\widehat h^{(2)}$ is
\begin{align}\label{eq:outer-curvature-flux}
\mathcal B_R
={}&\int_{S_R}\left[
\lambda(\divg\widehat h^{(2)}-d\tr\widehat h^{(2)})(\nu_S)
+(\tr\widehat h^{(2)})\partial_{\nu_S}\lambda
-\widehat h^{(2)}(\nabla\lambda,\nu_S)\right]\dd\sigma\\
&-\int_{E_R}\lambda\widehat h^{(2)}(\nu_\Gamma,\eta_\Gamma)\dd\sigma^{n-2}.
\notag
\end{align}
Here all derivatives, traces, normals, and measures are round, and
$\eta_\Gamma$ is the outward conormal to $E_R$ within $\Gamma_R$.
With $\rho=|y|$, define the Euclidean boundary term
\[
\mathcal F_R(z)=\int_{S_R^+}\left[
(\mathcal Dz)\partial_\rho z-\frac R2|\nabla z|^2
+\frac R2n(n+2)U^{4/(n-2)}z^2\right]\dd S.
\]

\begin{proposition}
For the data above,
\begin{equation}\label{eq:finite-pohozaev}
\mathcal P_R(h,h)=\frac{(n-2)^2}{2^{n+1}(n-1)}
\left[-Q_R(\widehat h)+\mathcal B_R\right]
+\mathcal F_R(z).
\end{equation}
If $4U^{4/(n-2)}h$ and $z/U$ extend smoothly to $\overline{\Omega_+}$,
then \eqref{eq:intro-pohozaev} holds.
\end{proposition}
\begin{proof}
For a weight $\varphi$, let $I_{\varphi,R}$ be $I_R$ with $\lambda$ replaced by $\varphi$,
and put
\[
\Phi_R(u)=\int_{S_R^+}\left[
(\mathcal Du)\partial_\rho u-\frac R2|\nabla u|^2
+\frac{R(n-2)^2}{2}u^{2n/(n-2)}\right]\dd S.
\]
Since $\det g_t=1$, conformal covariance gives, with
$\varphi_t=\mathcal Du_t/u_t$,
\begin{align*}
\frac{n-2}{2^n(n-1)}I_{\varphi_t,R}(\widehat g_t)
={}&\int_{B_R^+}(\mathcal Du_t)
\left(-L_{g_t}u_t-n(n-2)u_t^{(n+2)/(n-2)}\right)\dd y\\
&+\int_{D_R}(\mathcal Du_t)B_{g_t}u_t\,\omega_t\dd y'.
\end{align*}
Euclidean integration by parts gives
\[
\int_{B_R^+}(\mathcal Du)
\left(-\Delta u-n(n-2)u^{(n+2)/(n-2)}\right)\dd y
-\int_{D_R}(\mathcal Du)\partial_nu\dd y'=-\Phi_R(u).
\]
Subtracting yields
\[
\mathcal G_t^R(\mathcal Du_t,u_t)
=\frac{n-2}{2^n(n-1)}I_{\varphi_t,R}(\widehat g_t)+\Phi_R(u_t).
\]
Both curvature defects start at order $t^2$ by
\eqref{eq:yamabe-correction}, so replacing $\varphi_t$ by
$\varphi_0=\tfrac{n-2}{2}\lambda$ leaves the quadratic coefficient unchanged.
Moreover,
\[
[t^2]I_R(\widehat g_t)
=-Q_R(\widehat h)+DI_R(\gb)[\widehat h^{(2)}],\qquad
[t^2]\Phi_R(U+tz)=\mathcal F_R(z).
\]
Integration by parts gives $DI_R(\gb)[\widehat h^{(2)}]=\mathcal B_R$.
The face $\Gamma_R$ contributes only the integral over $E_R$ in
\eqref{eq:outer-curvature-flux}, since
$\partial_{\nu_\Gamma}\lambda=\mathrm{II}_{\gb}=0$ there.
Taking the coefficient of $t^2$ proves \eqref{eq:finite-pohozaev}.

Under the final hypotheses, $\widehat h$ and $\widehat h^{(2)}$ extend smoothly
across $-e_{n+1}$. The spherical areas of $S_R$ and $E_R$ are of orders
$R^{1-n}$ and $R^{2-n}$, respectively. Since the linearized curvature
defects vanish, $Q_R(\widehat h)$ is \eqref{eq:common-hessian} restricted
to $\Omega_R$ and $\Gamma_R$. Hence
\[
Q_R(\widehat h)\longrightarrow Q(\widehat h),\qquad
\mathcal B_R=O(R^{2-n})\longrightarrow0.
\]
Smoothness of $z/U$ gives $z=O(\rho^{2-n})$ and
$\nabla z=O(\rho^{1-n})$, so
$\mathcal F_R(z)=O(R^{2-n})\to0$ as well.
Taking $R\to\infty$ proves \eqref{eq:intro-pohozaev}.
\end{proof}


\begin{thebibliography}{99}
\raggedright

\bibitem{adn}
S.~Agmon, A.~Douglis, and L.~Nirenberg,
\emph{Estimates near the boundary for solutions of elliptic partial
differential equations satisfying general boundary conditions. I},
Comm. Pure Appl. Math. \textbf{12} (1959), no.~4, 623--727.

\bibitem{bce}
E.~Barbosa, M.~P.~Cavalcante, and J.~M.~Espinar,
\emph{Min-Oo conjecture for fully nonlinear conformally invariant equations},
Comm. Pure Appl. Math. \textbf{72} (2019), no.~11, 2259--2281.

\bibitem{bm}
S.~Brendle and F.~C.~Marques,
\emph{Scalar curvature rigidity of geodesic balls in $\Sph^n$},
J. Differential Geom. \textbf{88} (2011), no.~3, 379--394.

\bibitem{bmn}
S.~Brendle, F.~C.~Marques, and A.~Neves,
\emph{Deformations of the hemisphere that increase scalar curvature},
Invent. Math. \textbf{185} (2011), no.~1, 175--197.

\bibitem{corderomisteli}
S.~Cordero-Misteli,
\emph{Compactness and noncompactness for the positive Yamabe--Escobar
problem with minimal boundary},
arXiv:2610.08759v2 (2026).

\bibitem{cox}
G.~Cox,
\emph{Scalar curvature rigidity theorems for the upper hemisphere},
Ph.D. thesis, Duke University, 2011.

\bibitem{cmt}
G.~Cox, P.~Miao, and L.-F.~Tam,
\emph{Remarks on a scalar curvature rigidity theorem of Brendle and Marques},
Asian J. Math. \textbf{17} (2013), no.~3, 457--470.

\bibitem{ge}
J.~Ge,
\emph{Problems related to isoparametric theory},
in \emph{Surveys in Geometric Analysis 2019},
Science Press, Beijing, 2020, 71--85.

\bibitem{gkmwblowup}
L.~Gong, S.~Kim, M.~Musso, and J.~Wei,
\emph{On critical dimensions for compactness in the boundary Yamabe problem, I},
arXiv:2609.39945 (2026).

\bibitem{gkwq}
L.~Gong, S.~Kim, and J.~Wei,
\emph{Compactness and non-compactness theorems of the fourth- and
sixth-order constant $Q$-curvature problems},
arXiv:2502.14237v3 (2025).

\bibitem{hangwang}
F.~Hang and X.~Wang,
\emph{Rigidity and non-rigidity results on the sphere},
Comm. Anal. Geom. \textbf{14} (2006), no.~1, 91--106.

\bibitem{hohuang}
P.~T.~Ho and Y.-C.~Huang,
\emph{Deformation of the scalar curvature and the mean curvature},
arXiv:2008.11893 (2020).

\bibitem{hls}
Y.~Hu, P.~Liu, and Y.~Shi,
\emph{Rigidity of 3D spherical caps via $\mu$-bubbles},
Pacific J. Math. \textbf{323} (2023), no.~1, 89--114.

\bibitem{kms}
M.~A.~Khuri, F.~C.~Marques, and R.~M.~Schoen,
\emph{A compactness theorem for the Yamabe problem},
J. Differential Geom. \textbf{81} (2009), no.~1, 143--196.

\bibitem{kim}
S.~Kim,
\emph{Conformal scalar curvature rigidity on Riemannian manifolds},
arXiv:1706.00460 (2017).

\bibitem{leetam}
M.-C.~Lee and L.-F.~Tam,
\emph{Rigidity of Lipschitz map using harmonic map heat flow},
Amer. J. Math. \textbf{148} (2026), no.~4, 985--1002.

\bibitem{miaotam}
P.~Miao and L.-F.~Tam,
\emph{Scalar curvature rigidity with a volume constraint},
Comm. Anal. Geom. \textbf{20} (2012), no.~1, 1--30.

\bibitem{minoo}
M.~Min-Oo,
\emph{Scalar curvature rigidity of certain symmetric spaces},
in \emph{Geometry, Topology, and Dynamics} (Montreal, PQ, 1995),
CRM Proc. Lecture Notes, vol.~15, Amer. Math. Soc., Providence, RI, 1998,
127--136.

\bibitem{qingyuan}
J.~Qing and W.~Yuan,
\emph{On scalar curvature rigidity of vacuum static spaces},
Math. Ann. \textbf{365} (2016), no.~3--4, 1257--1277.

\bibitem{shengdeform}
H.~Sheng,
\emph{Localized deformation of the scalar curvature and the mean curvature},
arXiv:2402.08619 (2024).

\bibitem{shengstatic}
H.~Sheng,
\emph{Static manifolds with boundary and rigidity of scalar curvature and mean curvature},
Int. Math. Res. Not. IMRN (2025), no.~7, Paper No.~rnaf086.

\bibitem{spiegel}
F.~M.~Spiegel,
\emph{Scalar curvature rigidity for locally conformally flat manifolds with boundary},
arXiv:1511.06270v2 (2016).

\bibitem{sweeneyminoo}
P.~Sweeney Jr.,
\emph{New counterexamples to Min-Oo's conjecture via tunnels},
Proc. Amer. Math. Soc. \textbf{153} (2025), no.~4, 1771--1786.

\bibitem{yuan}
W.~Yuan,
\emph{Brown--York mass and compactly supported conformal deformations of scalar curvature},
J. Geom. Anal. \textbf{27} (2017), no.~1, 797--816.

\bibitem{zhangzhao}
Y.~Zhang and B.~Zhao,
\emph{Scalar curvature rigidity of small spherical caps with arbitrary fill-ins},
arXiv:2610.18413v1 (2026).

\end{thebibliography}
\end{document}